\documentclass[11pt,notitlepage,twoside,a4paper]{amsart}
 \usepackage{amsfonts}

\usepackage{amsmath,amssymb,enumerate}

\usepackage{epsfig,fancyhdr,color}

\usepackage{amssymb}
\usepackage{amsmath,amsthm}
\usepackage{amscd}
\usepackage{psfrag}
\usepackage{graphicx}
\usepackage[mathcal]{eucal}

\definecolor{NoteColor}{rgb}{1,0,0}

\renewcommand{\textsc}{\textcolor{red}}

\newtheorem{theorem}{\rm\bf Theorem}[section]
\newtheorem{proposition}[theorem]{\rm\bf Proposition}
\newtheorem{lemma}[theorem]{\rm\bf Lemma}

\newtheorem*{theorem 1}{\rm\bf Proposition 1}
\newtheorem*{theorem 2}{\rm\bf Proposition 2}

\theoremstyle{definition}

\theoremstyle{remark}

\def\interieur#1{\mathord{\mathop{\kern 0pt #1}\limits^\circ}}

\title[Measured foliations of the hexagon]{The space of measured foliations of the hexagon}

\author{Athanase Papadopoulos}
\address{Athanase Papadopoulos, Institut de Recherche Math{\'e}matique Avanc\'ee,
Universit{\'e} de Strasbourg and CNRS,
7 rue Ren\'e Descartes,
 67084 Strasbourg Cedex, France} \email{athanase.papadopoulos@math.unistra.fr}
\date{\today}

\author{Guillaume Th\'eret}
\address{Guillaume Th\'eret, Institut de Math\'ematiques de Bourgogne, 9 avenue Alain Savary 21078 Dijon, France}
\email{guillaume.theret71@orange.fr}

\date{\today}

\begin{document}

\begin{abstract}   
The theory of geometric structures on a surface with nonempty boundary can be developed by using a decomposition of such a surface into hexagons, in the same way as the theory of geometric structures on a surface without boundary is developed using the decomposition of such a surface into pairs of pants. The basic elements of  the theory for surfaces with boundary include the study of measured foliations and of hyperbolic structures on hexagons. It turns out that there is an interesting space of measured foliations on a hexagon, which is equipped with a piecewise-linear structure (in fact, a natural cell-decomposition), and this space is a natural  boundary for the space of hyperbolic structures with geodesic boundary and right angles on such a hexagon. In this paper, we describe these spaces and the related structures.

\medskip

\end{abstract}
\maketitle

\section{Introduction}
\label{intro}

Geometric structures on a surface with nonempty boundary can be studied by decomposing such a surface into hexagons, in the same way as geometric structures on a surface with boundary are usually studied using the decomposition of such a surface into pairs of pants. In the case of surfaces with nonempty boundary, the properly embedded arcs play an important role , analogous to the roles played by the simple closed curves in the theory of surfaces without boundary. This was used for instance in the paper \cite{PT1} and \cite{PT2}. The basic elements in the geometric theory of surfaces with boundary include measured foliations and hyperbolic structures on hexagons. It turns out that the space of measured foliations on a hexagon has a simple but interesting structure. It is equipped with a natural cell-decomposition and it is a natural boundary to the space of hyperbolic structures with geodesic boundary and right angles on the hexagon. This theory is developed in analogy with Thurston's boundary of the Teichm\"uller space of a closed surface. The same results for the case where the surface is a pair of pants follows immediately from the case of the hexagon.

\section{Measured foliations on the hexagon}
\label{s:measured}

(In all this paper, homotopies are relative to the boundary; that is, each boundary edge of the hexagon is fixed setwise.) 

\subsection{Definition}
We consider foliations on the hexagon, with isolated singularities with $3$ separatrices and such that no leaf is homotopic to a point.  There may be zero or one singularity in the hexagon (see Figure \ref{lines}). 
In particular, the five configurations in Figure \ref{six} are excluded.

\begin{figure}[ht!]
\psfrag{a}{$a$}
\psfrag{b}{$b$}
\psfrag{c}{$c$}
\psfrag{a'}{$a'$}
\psfrag{b'}{$b'$}
\psfrag{c'}{$c'$}
\centering
\includegraphics[width=.7\linewidth]{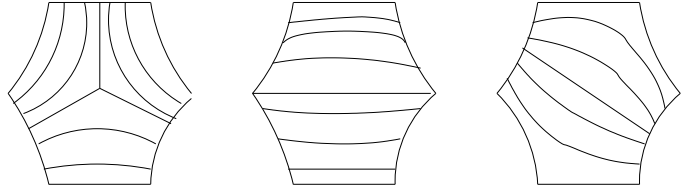}
\caption{\small{Foliations of the hexagon.}}
\label{lines}
\end{figure}

\begin{figure}[ht!]
\psfrag{a}{$a$}
\psfrag{b}{$b$}
\psfrag{c}{$c$}
\psfrag{a'}{$a'$}
\psfrag{b'}{$b'$}
\psfrag{c'}{$c'$}
\centering
\includegraphics[width=.9\linewidth]{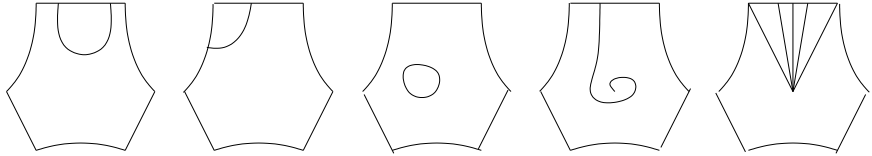}
\caption{\small{Any one of the following behavior of leaves is excluded.}}
\label{six}
\end{figure}

We consider a hexagon $H$ whose sides are denoted in cyclic order by $a, C, b, A,c,B$ ($a$ opposite to $A$, etc.)

An {\it arc} in the hexagon $H$ is the homeomorphic image of a closed interval,
whose interior is in the interior of $H$ and whose endpoints are on $\partial H$.

\subsection{Arc triples and local charts}

An \emph{arc triple} is a triple of disjoint arcs each joining two non-adjacent boundary edges of the hexagon. Up to homotopy, there are 14 arc triples in the hexagon, and some of them are represented in Figure \ref{triples}. The names of the other arcs are obtained similarly. In Figure \ref{triples}, the arc $\beta$ joins $b$ to $B$. The notation for the other arcs is analogous, with a cyclic change in names. That is, the arc $\alpha$ joins $a$ to $A$, and the arc $\gamma$ joins $c$ to $C$.

\begin{figure}[ht!]
\psfrag{a}{$a$}
\psfrag{b}{$b$}
\psfrag{c}{$c$}
\psfrag{A}{$A$}
\psfrag{B}{$B$}
\psfrag{C}{$C$}
\psfrag{p}{$\alpha$}
\psfrag{q}{$\beta$}
\psfrag{r}{$\gamma$}
\centering
\includegraphics[width=.8\linewidth]{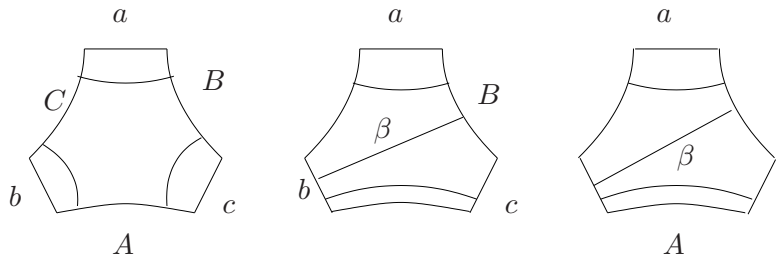}
\caption{\small{The arc triples in these three hexagons are respectively $\{a,b,c\}$, $\{a,\beta,c\}$ and $\{a,\beta,A\}$. The three cases represent Cases 1, 2 and 3 listed below.}}
\label{triples}
\end{figure}

Such a foliation $F$ is equipped with an invariant transverse measure, in the sense of \cite{FLP}. The \emph{geometric intersection number}, $i(F,\partial)$ of $F$ with an arc $\partial$ in the hexagon is the infimum of the transverse measure with $F$ of an arc homotopic to $\partial$.

We denote by $\mathcal{MF}$ the space of measured foliations of the hexagon up to homotopy. We equip  $\mathcal{MF}$ with the quotient topology induced by the geometric convergence of representatives together with convergence of transverse measures on arcs. We denote by $\mathcal{PMF}$ the quotient space of  $\mathcal{MF}$ by the natural action of the positive reals $\mathbb{R}_+$.

A measured foliation (or its equivalence class in $\mathcal{MF}$) is said to be \emph{in good position with respect to an arc triple $\{\partial_1,\partial_2,\partial_3\}$} if
\begin{enumerate}
\item $\displaystyle \sum_{j=1}^3 i(F,\partial_j)>0$ ;
\item $F$ has no leaf parallel to one of the three arcs in the triple $\{\partial_1,\partial_2,\partial_3\}$.
\end{enumerate}
The first condition is equivalent to the fact that at least one of the arcs $\partial_1,\partial_2,\partial_3$ has positive $F-$measure.

An example is given in Figure \ref{example}.

\begin{figure}[ht!]
\psfrag{a}{$a$}
\psfrag{b}{$b$}
\psfrag{c}{$c$}
\psfrag{A}{$A$}
\psfrag{B}{$B$}
\psfrag{C}{$C$}
\psfrag{p}{$\alpha$}
\psfrag{q}{$\beta$}
\psfrag{r}{$\gamma$}
\centering
\includegraphics[width=.25\linewidth]{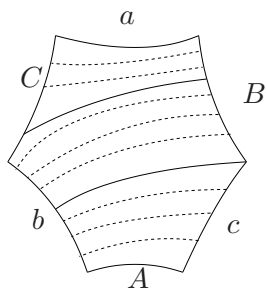}
\caption{\small{The foliation represented is neither in good position with respect to $\{a,b,c\}$ nor to $\{A,B,C\}$ but it is in good position with respect to $(\alpha,B,C)$, $(c,C,\alpha)$, $(b,B,\alpha)$, etc.}}
\label{example}
\end{figure}

The behavior of measured foliations in good position with respect to an arc triple $\{\partial_1,\partial_2,\partial_3\}$ is conveniently described by the projective space of transverse measures, see Figure \ref{proj}.

\begin{figure}[ht!]
\psfrag{1}{$\partial_1$}
\psfrag{2}{$\partial_2$}
\psfrag{3}{$\partial_3$}
\psfrag{4}{$\displaystyle \sum_{j=1}^3\partial_j=1$}
\centering
\includegraphics[width=.4\linewidth]{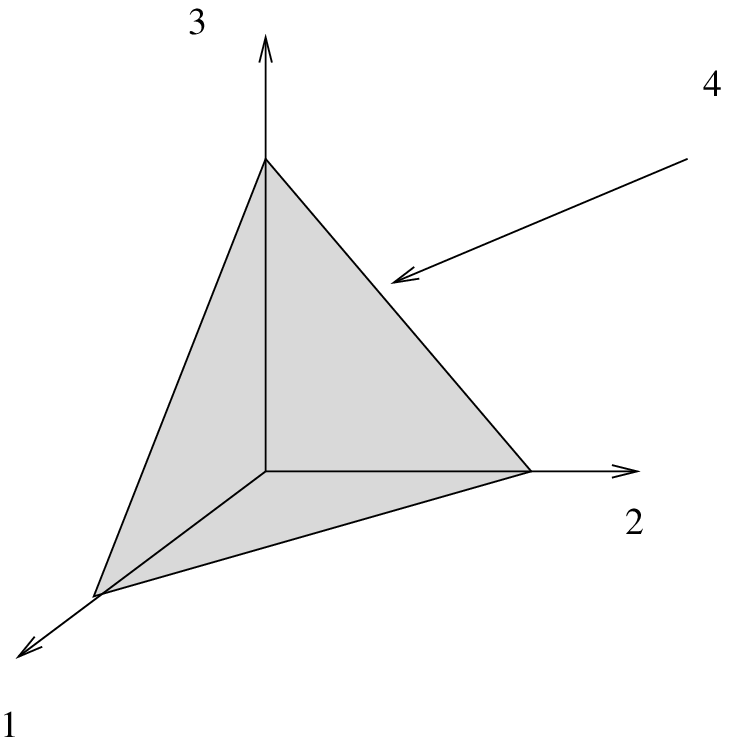}
\caption{\small{}}
\label{proj}
\end{figure}

In what follows, we shall often draw \emph{partial foliations}, that is, foliations whose support is a subsurface of the hexagon; it will be clear from the picture how one has to collapse the non-foliated regions in order to obtain a measured foliation which is well-defined up to equivalence.
 
We shall denote by $\mathcal{MF}_\partial$ and $\mathcal{PMF}_\partial$ the subspaces of $\mathcal{MF}$ and $\mathcal{PMF}$ respectively of equivalence classes of measured foliations that are in good position with respect to an arc triple $\partial$.

We have $\mathcal{MF}= \mathcal{MF}_{\{a,b,c\}}\cup \mathcal{MF}_{\{A,B,C\}}\cup\mathcal{MF}_{\{a,\beta,c\}}\cup\mathcal{MF}_{\{a,b,\gamma\}}\cup \ldots$

We shall gather the subsets $\mathcal{PMF}_\partial$ associated to the 14 possible arc triples in the following three categories:

\noindent {\bf Case 1}:  $\{a,b,c\}$, $\{A,B,C\}$.
Figure \ref{des4} represents the subset $\mathcal{PMF}_\partial$ with $\partial=\{a,b,c\}$. There is a similar description for $\partial=\{A,B,C\}$. This case corresponds to the hexagon to the left in Figure \ref{triples}. 
\begin{figure}[ht!]
\psfrag{a}{$a$}
\psfrag{b}{$b$}
\psfrag{c}{$c$}
\psfrag{1}{\tiny $a=b>0; \ c=0$}
\psfrag{2}{\tiny $a=0; \ b>c>0$}
\psfrag{3}{\tiny $\frac{1}{2}(a+b-c)$}
\psfrag{4}{\tiny $\frac{1}{2}(a+c-b)$}
\psfrag{5}{\tiny  $\frac{1}{2}(b+c-a)$}
\psfrag{6}{\tiny $c-a-b$}
\psfrag{7}{\tiny $c>a+b \ (a<c; b<c)$}
\centering
\includegraphics[width=.7\linewidth]{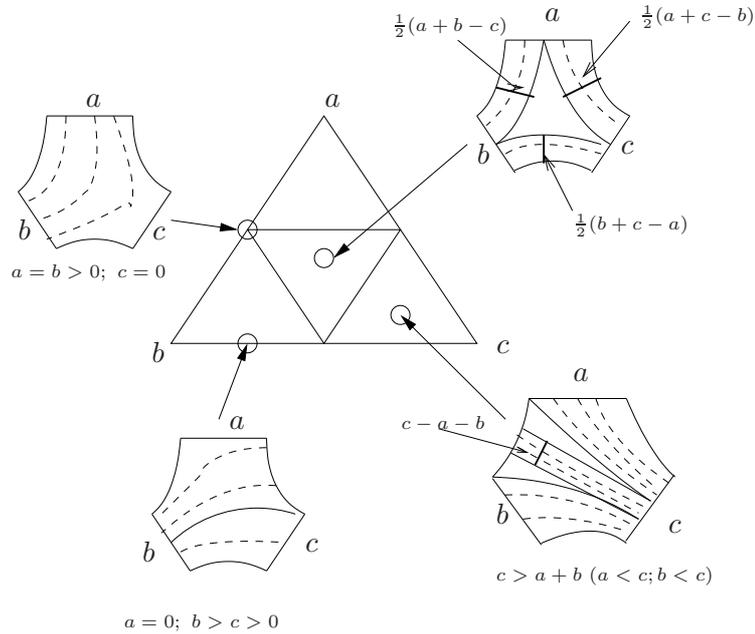}
\caption{\small{The patch $\mathcal{PMF}_{\{a,b,c\}}$} together with four points, two in the interior and two on the boundary. The patch is subdivided into four parts; the central triangle corresponds to the triples $a,b,c$ satisfying the triangle inequality. The leaves of the foliations belonging to each subpart have the same behavior; only the transverse measures change.}
\label{des4}
\end{figure}

\noindent {\bf Case 2}: $\{\alpha,b,c\}$, $\{a,\beta,c\}$, $\{a,b,\gamma\}$, $\{\alpha,B,C\}$, $\{A,\beta,C\}$, $\{A,B,\gamma\}$.
Figure \ref{des5} represents the subset $\mathcal{PMF}_\partial$ with $\partial= \{\alpha, b,c\}$ (the hexagon to the middle in Figure \ref{triples}). The triangle $\alpha bc$ is cut into four regions.
\begin{figure}[ht!]
\psfrag{a}{$a$}
\psfrag{b}{$b$}
\psfrag{c}{$c$}
\psfrag{q}{$\alpha$}
\psfrag{1}{\tiny $b=c=0$}
\psfrag{2}{\tiny $\alpha-b$}
\psfrag{3}{\tiny $b-c$}
\psfrag{4}{\tiny $\alpha>b>c$}
\psfrag{5}{\tiny $\alpha=c=0$}
\psfrag{6}{\tiny $\alpha=b=c$}
\psfrag{7}{\tiny $c>b>\alpha$}
\psfrag{8}{\tiny $c>\alpha>b$}
\psfrag{9}{\tiny $c$}
\psfrag{x}{\tiny $c-b$}
\psfrag{y}{\tiny $c-\alpha$}
\centering
\includegraphics[width=.85\linewidth]{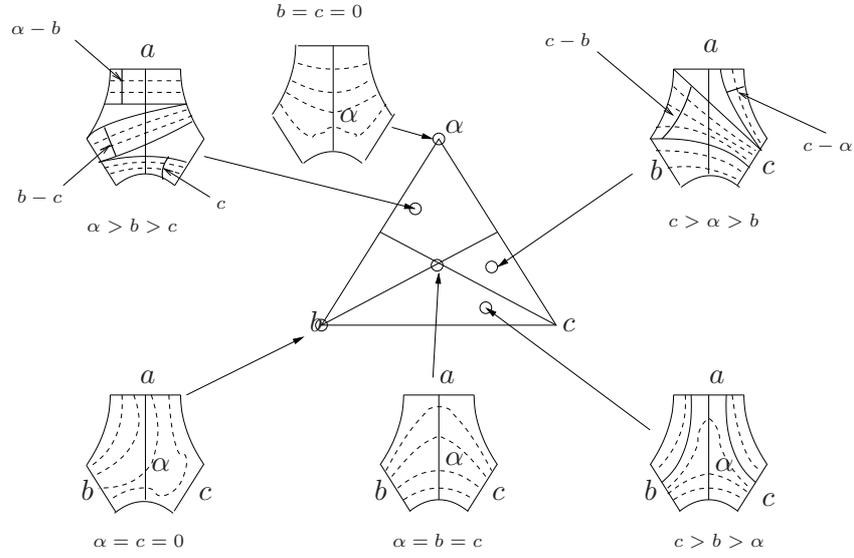}
\caption{\small{The patch $\mathcal{PMF}_{\{\alpha,b,c\}}$. It is subdivised into four parts. As before, each subdivision of the patch corresponds to a particular behavior of the leaves of the foliations.}}
\label{des5}
\end{figure}

\noindent  {\bf  Case 3}: $\{a,A,\beta\}$, $\{b,B,\alpha\}$, $\{c,C\alpha\}$, $\{a,A,\gamma\}$, $\{b,B,\gamma\}$, $\{c,C,\beta\}$.
Figure \ref{des6} represents the subset $\mathcal{PMF}_\partial$ with $\partial= \{a,A,\beta\}$ (the hexagon to the right in Figure \ref{triples}). 
\medskip

\begin{figure}[ht!]
\psfrag{a}{$a$}
\psfrag{b}{$\beta$}
\psfrag{A}{$A$}
\psfrag{2}{\tiny $A>\alpha>\beta$}
\psfrag{1}{\tiny $a>\beta>A$}
\psfrag{3}{\tiny $\beta >A>a$}
\psfrag{4}{\tiny $A>\beta>a$}
\centering
\includegraphics[width=.8\linewidth]{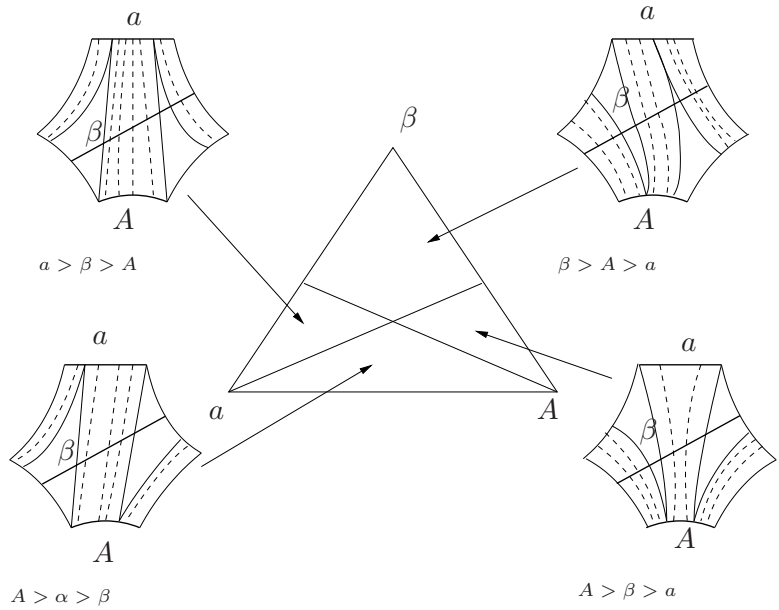}
\caption{\small{The patch $\mathcal{PMF}_{\{a,A,\beta\}}$ is represented together with a few points.}}
\label{des6}
\end{figure}

\subsection{Projective measured foliations is a sphere with a piecewise-linear structure}

The subsets $\mathcal{PMF}_\partial$ are local charts for $\mathcal{PMF}$ forming an atlas for this space.
We shall show that the coordinate changes between local charts are piecewise-linear.
This will show that $\mathcal{PMF}$ is a piecewise-linear manifold. 
We shall see that this manifold is homeomorphic to the 2-sphere.

In Figure \ref{phi}, we have represented the coordinate changes. The maps $\phi_i$ are the coordinate change maps between the charts. In this figure, the charts are represented in the projective space, but they have to be thought of as being in $\mathcal{MF}$. The charts are triangular, and there are two sorts of coordinate changes: 
\begin{enumerate}
\item  The intersection of the two charts (the domain and range of the map $\phi_i$) has nonempty interior. In this case the intersection is left blank in the figure.
\item The intersection of the two charts has empty interior. In this case, the intersection is along the boundary, and in the figure this intersection is represented in thicker lines.
\end{enumerate}

\begin{figure}[h!]
\psfrag{m}{\tiny $n$}
\psfrag{w}{\tiny $m$}
\psfrag{p}{\tiny $p$}
\psfrag{a}{\tiny $a$}
\psfrag{b}{\tiny $b$}
\psfrag{c}{\tiny $c$}
\psfrag{i}{\tiny $i$}
\psfrag{j}{\tiny $j$}
\psfrag{k}{\tiny $k$}
\psfrag{1}{\tiny $1$}
\psfrag{2}{\tiny $2$}
\psfrag{3}{\tiny $3$}
\psfrag{f}{\tiny $\phi_1$}
\psfrag{g}{\tiny $\phi_2$}
\psfrag{h}{\tiny $\phi_3$}
\psfrag{5}{\tiny $\phi_4$}
\psfrag{6}{\tiny $\phi_5$}
\psfrag{A}{\tiny $A$}
\psfrag{B}{\tiny $B$}
\psfrag{C}{\tiny $C$}
\psfrag{J}{\tiny $j$}
\psfrag{t}{\tiny $\alpha=\phi_4(p)$}
\psfrag{u}{\tiny $\phi_4(a)$}
\psfrag{v}{\tiny $\phi_4(n)$}
\psfrag{q}{\tiny $\alpha$}
\psfrag{w}{\tiny $\phi_3(\alpha)$}
\psfrag{o}{\tiny $\phi_3(k)$}
\psfrag{M}{\tiny $\alpha$}
\psfrag{g}{\tiny $\phi_2(c)$}
\psfrag{h}{\tiny $\phi_2(b)$}
\psfrag{T}{\tiny $\phi_2(j)$}
\psfrag{Y}{\tiny $c=\phi_2(k)$}
\psfrag{Z}{\tiny $b=\phi_2(i)$}
\psfrag{P}{\tiny $\alpha=\phi_1(j)$}
\psfrag{Q}{\tiny $c=\phi_1(b)$}
\psfrag{R}{\tiny $b=\phi_1(c)$}
\psfrag{W}{\tiny $q$}
\psfrag{X}{\tiny $a=\phi_3(i)$}
\psfrag{H}{\tiny $C=0$}
\psfrag{I}{\tiny $B>0$}
\psfrag{D}{\tiny $A=0$}
\psfrag{x}{\tiny $\alpha=\phi_5(j)$}
\psfrag{y}{\tiny $\phi_5(C)$}
\psfrag{z}{\tiny $\phi_5(B)$}
\psfrag{11}{\tiny $\phi_2$}
\psfrag{12}{\tiny $\phi_3$}
\psfrag{13}{\tiny $\phi_4$}
\centering
\includegraphics[width=.8\linewidth]{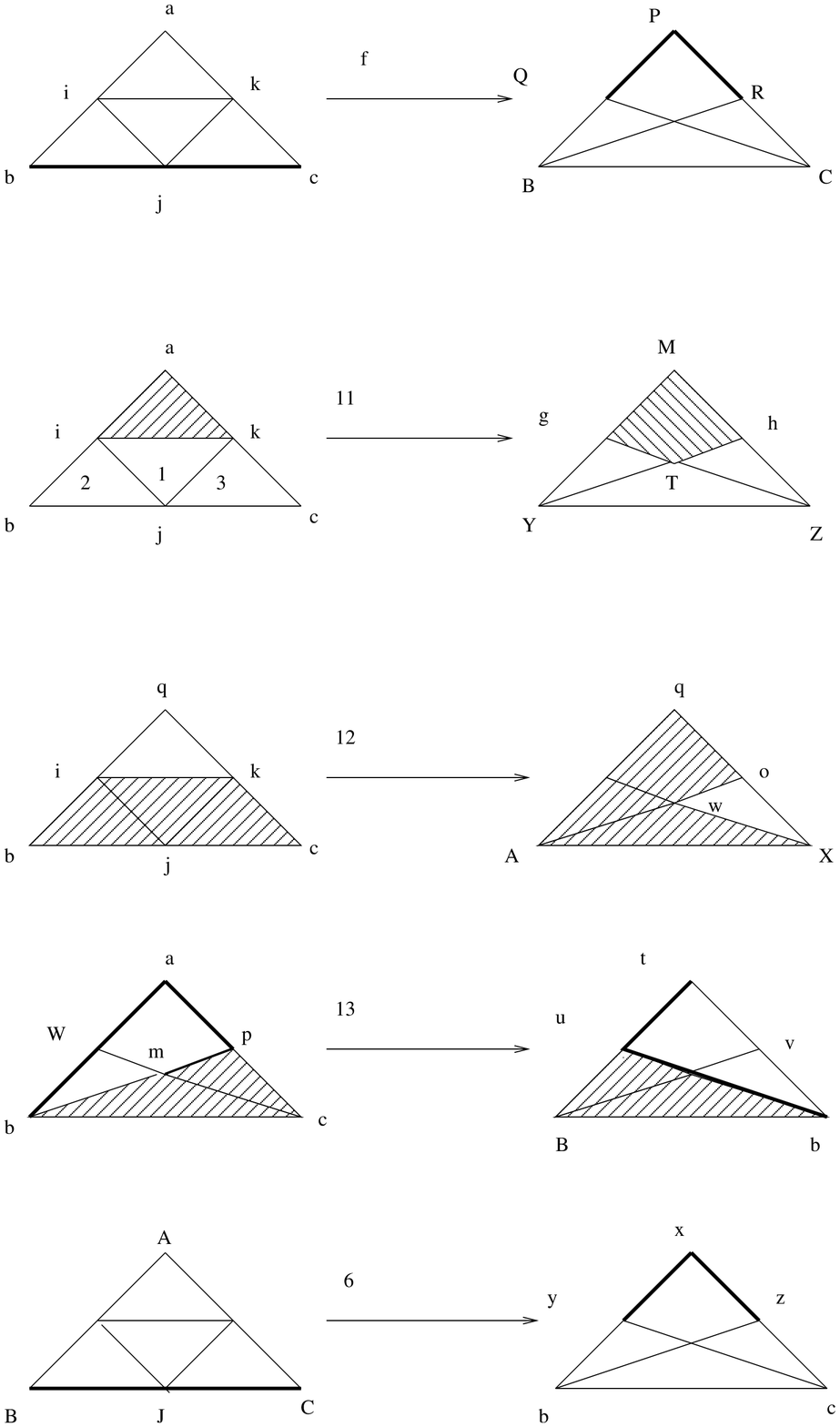}
\caption{\small{}}
\label{phi}
\end{figure}

From the following formulae for the coordinate changes we can see that they are piecewise-linear.
There are five cases, corresponding to the cases drawn in Figure \ref{phi}. 
Let us consider as an example the case of $\phi_{2}$: this is the coordinate change map which maps the non-hachured part of the triangle $(a,b,c)$ on the left to the non-hachured part of the triangle $(\alpha,b,c)$ on the right. 
These two pieces coincide as subsets of $\mathcal{MF}$ since they contain the same foliations.
The non-hachured part of the triangle $(a,b,c)$ is subdivised into three triangular subparts (numbered 1, 2, 3). 
The map $\phi_{2}$ is linear on each of these subparts in the parameters $a,b,c$ (see the corresponding formulae, in which $(\alpha,b,c)$ stand for the parameters in the image triangle). The map
$\phi_{2}$ is therefore piecewise-linear on the non-hachured part of the triangle $(a,b,c)$ on the left. 
The same holds for the other triangles.

\begin{enumerate}
\item Formula for $\phi_{1}$: If $b\geq c$, then
$  \begin{cases}\alpha = b\\
B=b-c\\
C=0
\end{cases}
$
and if $b\leq c$,  then $  \begin{cases}\alpha = c\\
B=0\\
C=c-b
\end{cases}
$

\item Formula for $\phi_{2}$: If $(a,b,c)$ is in the region labelled 1, then
$  \begin{cases}\alpha = \frac{1}{2}(b+c-a)\\
b=b\\
c=c
\end{cases}
$

\noindent and if $(a,b,c)$ is in the region labelled 2, then
$  \begin{cases}\alpha =b-a\\
b=b\\
c=c
\end{cases}
$

\noindent and if $(a,b,c)$ is in the region labelled 3, then
$  \begin{cases}\alpha =c-a\\
b=b\\
c=c
\end{cases}
$

\item Formula for $\phi_{3}$: If $a\geq b+c$ , then
$  \begin{cases}\beta=a-b\\
a=a\\
A=a-b-c
\end{cases}
$ ($a>\beta > A$).

\item Formula for $\phi_{4}$: If $\alpha>c$ , then
$  \begin{cases}\alpha = \alpha\\
b=b\\
B=\alpha-c 
\end{cases}
$

\item Formula for $\phi_{5}$: If $C>B$ and $A=0$, then
$  \begin{cases}\alpha = C\\
b=0\\
c=C-B
\end{cases}
$ ;

\noindent   if  $C<B$ and $A=0$, then
$  \begin{cases}\alpha =B\\
b=B-C\\
c=0
\end{cases}
$

\end{enumerate}

Figure \ref{sphere} represents the space $\mathcal{PMF}$, which is homeomorphic to a sphere, represented in the plane with a point at infinity.
The small star $ABC$ is the chart $\mathcal{PMF}_{\{a,b,c\}}$, and the exterior of the big star $IAKCJB$ is the chart $\mathcal{PMF}_{\{A,B,C\}}$. Its centre is at infinity.

The other charts are polygons whose edges are contained in the graph that is drawn.

Using this and the fact that $\mathcal{PMF}_{\{a,b,c\}}\cap \mathcal{PMF}_{\{A,B,C\}}$ is constituted of the three vertices of each triangle, Figure \ref{sphere} shows that $\mathcal{PMF}$ is a sphere equipped with a natural piecewise-linear structure. In fact, more than that, it is equipped with a natural triangulation. In Figure \ref{sphere}, the centre of the triangle $\mathcal{PMF}_{\{A,B,C\}}$ is at infinity and the centre of the triangle $\mathcal{PMF}_{\{a,b,c\}}$ is at 0.

\begin{figure}[ht!]
\psfrag{C}{\tiny $C$}
\psfrag{B}{\tiny $B$}
\psfrag{A}{\tiny $A$}
\psfrag{I}{\tiny $I$}
\psfrag{J}{\tiny $J$}
\psfrag{K}{\tiny $K$}
\psfrag{M}{\tiny $M$}
\psfrag{P}{\tiny $P$}
\psfrag{N}{\tiny $N$}
\centering
\includegraphics[width=.6\linewidth]{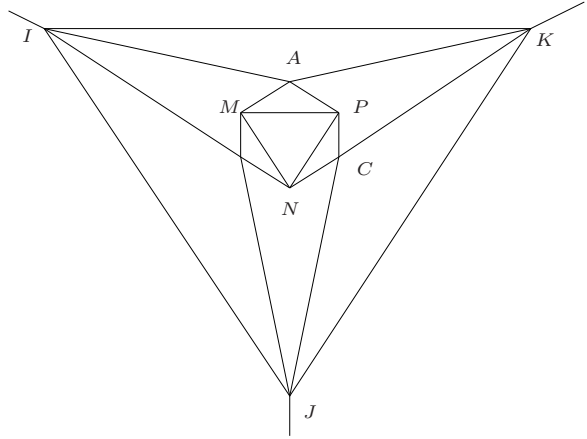}
\caption{\small{The sphere $\mathcal{PMF}$ represented on the plane with one point at infinity, i.e., through stereographic projection.}}
\label{sphere}
\end{figure}

\section{Hyperbolic structures on the right hexagon and compactification}

In this paper, a \emph{right hexagon} is a hexagon equipped with a hyperbolic metric such that all the edges are geodesic and all angles are right.

\subsection{The lengths of arc triples determine the hyperbolic structure.}

We denote by $\mathcal{T}$ the Teichm\"uller space of the right hexagon. 
As we shall presently prove, a point in $\mathcal{T}$ is determined by the three lengths of any fixed arc triple, for instance $\{a,b,c\}$. 

\begin{proposition} \label{prop:hexagon}
A right hexagon is determined by the three lengths of an arbitrary arc triple.
\end{proposition}

\begin{proof}
We use the notation of Figure \ref{triples}. 
The result is well known for the triple $\{a,b,c\}$ or $\{A,B,C\}$ (see the trigonometric formula in Fenchel \cite{Fenchel} p. 85 which shows that the isometry type of a right hexagon is determined by the length of any triple of non-consecutive edges; see also Thurston \cite{Thurston}).
   
Let us check the result for the triple $\{a,b,\gamma\}$. 
We consider the right pentagon of Figure \ref{prop} (1). 
We have (see \cite{Fenchel}) 
\[\cosh t=\frac{\cosh a \cosh \gamma+\cos \phi(t)}{\sinh a \sinh \gamma}.\]
When $a$ and $\gamma$ are fixed, we obtain 
$\cos \phi(t)=K\cosh t -L$, where $K,L>0$ are constants.
Thus, there is a unique solution to $\phi(t)=\pi_2$.
Therefore, a right pentagon is determined by the length of two disjoint edges $\{a,\gamma$\} (Figure \ref{prop} (2)).

Thus, the right hexagon is determined by $\{a,b,\gamma\}$ since $\{a,\gamma\}$ and $\{b,\gamma\}$ determine the two right pentagons having $\gamma$ as an edge (Figure \ref{prop} (3)). 

The remaining case $\{a,A,\gamma\}$ is treated in the same way since $\{a,\gamma\}$ and $\{A,\gamma\}$ determine two right pentagons (Figure \ref{prop} (4)).

\begin{figure}[ht!]
\psfrag{a}{$a$}
\psfrag{f}{$\phi(t)$}
\psfrag{g}{$\gamma$}
\psfrag{b}{$b$}
\psfrag{A}{$A$}
\psfrag{1}{$(1)$}
\psfrag{2}{$(2)$}
\psfrag{3}{$(3)$}
\psfrag{4}{$(4)$}
\psfrag{t}{$t$}
\centering
\includegraphics[width=.9\linewidth]{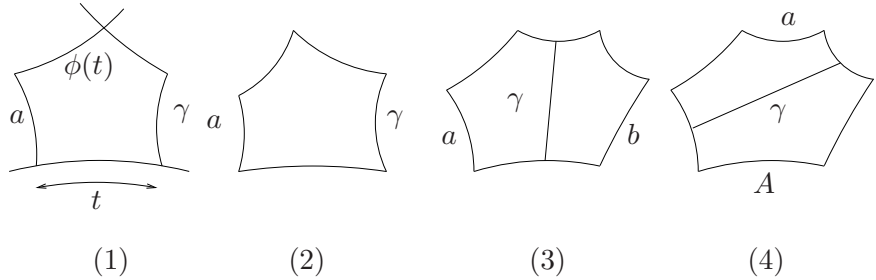}
\caption{\small{Here, all the angles, with possible exception of $\phi(t)$, are right angles.}}
\label{prop}
\end{figure}  
\end{proof}

Let $\partial=\{\partial_{1},\partial_{2},\partial_{3}\}$ be a fixed arc triple of the right hexagon.
Let 
\begin{align*}
l_{*}^{\partial}\ :\ \mathcal{T}\to& \mathbb{R}^{3}_{+}\\
h\mapsto&\big{(}l_{h}(\partial_{1}),l_{h}(\partial_{2}),l_{h}(\partial_{3})\big{)}
\end{align*}
be the length functional associated to $\partial$.

\begin{proposition} 
The length functions of any arc triple determine a homeomorphism $\mathcal{T}\to \mathbb{R}^3_+$.
\end{proposition}

\begin{proof} 
Proposition \ref{prop:hexagon} shows that $l_*^\partial : \mathcal{T}\to \mathbb{R}^3_+$ is injective. 
The surjectivity also follows from the proof of Proposition \ref{prop:hexagon}, where we see that we can choose the lengths of three consecutive edges independently from each other. 
Continuity is clear.
\end{proof}

\subsection{Converging to infinity in Teichm\"uller space.}

For a coming proof, we recall right here the following well-known trigonometric formulae, \cite{Fenchel} p. 85 and 86. (We use the notation of Figure \ref{trigo}).

\begin{figure}[ht!]
\psfrag{a}{$a$}
\psfrag{b}{$b$}
\psfrag{c}{$c$}
 \psfrag{A}{$A$}
 \psfrag{B}{$B$}
\psfrag{C}{$C$}
\psfrag{x}{$x$}
\psfrag{y}{$y$}
\psfrag{p}{$\alpha$}
\centering
\includegraphics[width=.4\linewidth]{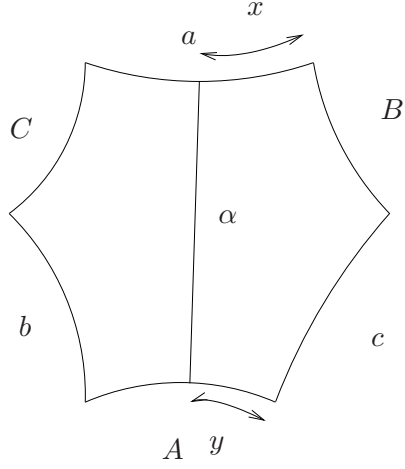}
\caption{\small{In this figure, all angles are right angles.}}
\label{trigo}
\end{figure}

\begin{align*}
\cosh C=&\frac{\cosh c +\cosh a \cosh b}{\sinh a \sinh b}\\
\frac{\sinh A}{\sinh a}=&\frac{\sinh B}{\sinh B}=\frac{\sinh C}{\sinh c}\\
\cosh \alpha =&\frac{1}{\tanh x}\times \frac{1}{\tanh y}=\frac{1}{\tanh (a-x)}\times \frac{1}{\tanh (A-y)}.
\end{align*}


Let us collect some asymptotic behaviors of the lengths of the sides of a right hexagon.

\begin{lemma} \label{lem:if}
Consider a right hexagon with consecutive sides $A,c,B,a,C,b$ and let $\alpha$ be the unique geodesic segment
joining $A$ and $a$ perpendicularly.
The following asymptotic relations hold.
\begin{enumerate}
\item \label{lem:if1} If $a\to 0$ then $\alpha\to\infty$.
\item \label{lem:if2} If $A\to 0$ then $\alpha\to\infty$.
\item \label{lem:if3} If $a\to 0$ and $b\leq K<\infty$ then $C\to \infty$. 
\item \label{lem:if4} If $a\to\infty$, $b\to \infty$ and $c\leq K< \infty$, then, $C\to 0$. \\
More generally,
if $a\to\infty$, $b\geq \epsilon_0>0$ and $c\leq K<\infty$, then $C\to 0$. 
\end{enumerate}
\end{lemma}
 
\begin{proof}
For (\ref{lem:if1}), since $x\to 0$, we have 
\[ \cosh \alpha = \frac{1}{\tanh x}\times \frac{1}{\tanh y}\geq \frac{1}{\tanh x}\to\infty.\]
 
For (\ref{lem:if3}), we have 
\[\cosh C=\frac{\cosh c}{\sinh a\sinh b} + \frac{1}{\tanh a\tanh b}\geq \frac{1}{\tanh a\tanh b}\to\infty.\] 
 
The other statements are also easy.
\end{proof}

We shall say that a sequence $(x_n)$ in $\mathcal{T}$ \emph{tends to infinity} if there exists an arc $\partial_{j}$
of an arc triple $\partial$ such that
$\vert \log l_{x_{n}}(\partial_j)\vert\to\infty$ as $n\to\infty$.

\begin{lemma}
\label{lemma:infinity}
A sequence $(x_n)$ in $\mathcal{T}$ tends to infinity if and only if there exists an arc triple $\partial=\{\partial_1,\partial_2,\partial_3\}$ such that for all $j=1,2,3$ we have $ l_{x_{n}}(\partial_j)\to\infty$ as $n\to\infty$.
\end{lemma}

\begin{proof}
We shall denote the asymptotic behavior of the triple $\{\partial_1,\partial_2,\partial_3\}$ by $\{K,0,\infty\}$ to express the fact that there exist distinct $i,j,k\in\{1,2,3\}$ such that $l_{x_{n}}(\partial_i)$ is bounded by $K$, $l_{x_{n}}(\partial_j)\to 0$ and $l_{x_{n}}(\partial_k)\to\infty$ as $n\to\infty$.

Let us first deal with the following asymptotic behaviors of $\{a,b,c\}$, which constitute all possibilities for convergence at infinity: $\{0,0,0\}$; $\{\infty,\infty,\infty\}$; $\{\infty,\infty,K\}$; $\{\infty,K,K\}$; $\{K,K,0\}$. We use statements (\ref{lem:if1})-(\ref{lem:if4}) of Lemma \ref{lem:if}.

\begin{enumerate}[--]
\item Case $\{a,b,c\}\to\{0,0,0\}$: Applying three times (\ref{lem:if3}), we have $(\{a,b,c\}\to \{0,0,0\}) \Rightarrow (\{A,B,C\}\to\{\infty,\infty;\infty\})$.

\item Case $\{a,b,c\}\to\{\infty,\infty,\infty\}$: nothing is needed in this case.

\item Case $\{a,b,c\}\to\{\infty,\infty,K\}$: From (\ref{lem:if4}), we have $C\to 0$. From (\ref{lem:if2}), we have $\gamma\to\infty$. Thus, if $a\to\infty$, $b\to\infty$ and $c\to K$, then $\{a,b,\gamma\}\to\{\infty,\infty,\infty\}$.

\item Case $\{a,b,c\}\to\{\infty,K,K\}$: From (\ref{lem:if4}), we have $C\to 0$. From (\ref{lem:if2}), we have $\gamma\to\infty$. Then, $\cosh A=\displaystyle \frac{\cosh a}{\sinh b\sinh c}+\frac{1}{\tanh b\tanh c}\to\infty$. Thus, $\{a,A,\gamma\}\to \{\infty,\infty,\infty\}$.

\item Case $\{a,b,c\}\to \{K,K,0\}$. Using (\ref{lem:if3}) two times, we get $B\to\infty$ and $A\to\infty$. From (\ref{lem:if1}) we get $\gamma\to\infty$. Thus, $\{A,B,\gamma\}\to \{\infty,\infty,\infty\}$. 

\end{enumerate}

The above reasoning also holds throughout for $\partial=\{A,B,C\}$. 

Suppose now that we have convergence to infinity with respect to the triple $\partial=\{a,b,\gamma\}$. 
Then, either convergence to infinity also holds for $\{a,b,c\}$, or not, and in that case, up to taking a subsequence, we can assume that $a<\infty$, $b<\infty$ and $c<\infty$. But in that case, the sequence stays in a compact set of $\mathcal{T}$, which is not possible.
The same reasoning applies for $\partial = \{a,A,\gamma\}$ and we have exhausted all cases. 
\end{proof}

\subsection{Embedding in the projective space}

The space $\mathcal{MF}$ is embedded in $\mathbb{R}^6_{\geq 0}$ using the intersection functional:
\begin{align*}
i_*: \mathcal{MF}\hookrightarrow& \mathbb{R}^6_{\geq 0}\\
h\mapsto& \left(i(F,a), i(F,b), i(F,c),i(F,A),i(F,B),i(F,C)\right).
\end{align*} 
 
The space $\mathcal{T}$ is embedded in $\mathbb{R}^6_{> 0}$ using the length functional:
\begin{align*}
l_*: \mathcal{T}\hookrightarrow& \mathbb{R}^6_{>0}\\
h\mapsto& \left(l_h(a), l_h(b), l_h(c),l_h(A),l_h(B),l_h(C)\right).
\end{align*} 
 
The images $l_*(\mathcal{T})$ and $i_*(\mathcal{MF})$ are disjoint since for all $F\in\mathcal{MF}$ there is at least one component of $i_*(F)$ in $\mathbb{R}^6_{\geq 0}$ which is zero.
 
\begin{proposition}
The natural map from the space $\mathcal{T}$ to ${\rm P}\mathbb{R}^6_{\geq 0}$ is injective.
\end{proposition}

\begin{proof}
Let $f:\mathbb{R}_{>0}\to \mathbb{R}$ be defined by \[t\mapsto \frac{\cosh (tc)}{\sinh (ta) \sinh (tb)},\] where $a\geq b\geq c>0$. 
The sign of  $f'(t)$ is that of the function
\begin{eqnarray*}
u(t)&=&c \sinh (tc) \sinh (ta) \sinh (tb) 
\\&{}& -a \cosh (ta) \cosh (tc) \sinh (tb) - b\cosh (tb) \cosh (tc) \sinh (ta).
\end{eqnarray*}
We write $u(t)=p(t)-n(t)$, where $p(t)$ is the first term (positive) in the above expression for $u(t)$ and $n(t)$ is the opposite of the second and third terms. 
We have 
\begin{eqnarray*}
n(t)&=& \cosh (tc) \left(a \cosh (ta) \sinh (tb) + b\cosh (tb) \sinh (ta)\right)\\
&\geq & \min \{a,b\} \cosh (tc) \left(\cosh (ta) \sinh (tb) + \cosh (tb) \sinh (ta)\right)\\
&\geq & b\cosh (tc) \sinh (t(a+b)).
\end{eqnarray*}
Therefore,
\begin{eqnarray*}
u(t)&\leq& c\sinh (tc) \sinh (ta) \sinh (tb) - b\cosh (tc) \sinh (t(a+b))\\
&\leq & b\left( \sinh (tc) \sinh (ta) \sinh (tb) -\cosh (tc) \sinh (t(a+b))\right) \hbox { (since $b\geq c$) } \\
&\leq &  b\left( \sinh (tc) \left(\cosh (t(a+b))-\cosh (ta) \cosh (tb)\right) - \cosh (tc) \sinh (t(a+b))\right)\\
&\leq &   -b\sinh (t(a+b-c))-b\sinh (tc)\cosh (ta) \cosh (tb).
\end{eqnarray*}
Since $a+b-c>0$, we obtain $u(t)<0$.
Thus, the function $f(t)$ is strictly decreasing on $[0,\infty[$.

Let $h$ be now a right hexagon and let $t\geq 0$. 
Up to permuting $a,b,c$, we can assume that the lengths satisfy $a\geq b\geq c>0$.

For $t>0$, let $h_t$ be the hexagon obtained by multiplying the lengths $a,b,c$ by the factor $t$. 
Then,
\[ \cosh C(t)=\frac{\cosh (tc)+\cosh (ta)\cosh (tb)}{\sinh (ta)\sinh (tb)}= f(t) +\frac{1}{\tanh (ta)}\times \frac{1}{\tanh (tb)}.\]

The map $t\mapsto 1/{\tanh (ta)}$ is strictly decreasing, therefore the map 
\[t\mapsto \frac{1}{\tanh (ta)}\times \frac{1}{\tanh (tb)}\] is strictly decreasing and $f(t)$ is also strictly decreasing. Thus, since $x\mapsto \cosh x$ is strictly increasing on $[0,\infty]$, we deduce that $t\mapsto C(t)$ is strictly decreasing.

In conclusion, we obtain the following : if we expand the lengths of three non-consecutive edges of the right hexagon by the same factor, then one of the three other edges is contracted.
As a consequence, there are no homothetic right hexagons. 
This completes the proof of the proposition.
\end{proof}

We denote by $\pi$ the projection map : $\mathbb{R}^6\setminus \{0\}\to {\rm P}\mathbb{R}^6.$
Let
\[\overline{\mathcal{T}} := \pi\circ l_*(\mathcal{T})\cup \pi\circ i_*(\mathcal{MF}).\]

By what precedes, $\pi\circ l_*(\mathcal{T})$ can be identified with Teichm\"uller space $\mathcal{T}$ itself while $\pi\circ i_*(\mathcal{MF})$ can be naturally identified with $\mathcal{PMF}$.

\subsection{Projection on $\mathcal{MF}_\partial$}

Let $\epsilon >0$ and let $\partial=\{\partial_1,\partial_2,\partial_3\}$ be an arc triple.
Set
\[\mathcal{T}^{\epsilon}_\partial := \{h\in\mathcal{T}\ \vert \ l_h(\partial_j)>\epsilon,\ \forall j=1,2,3\}.\]
This subset is open in $\mathcal{T}$.

\begin{proposition} 
There exists a positive $\epsilon$ such that the patches $\mathcal{T}^{\epsilon}_\partial$ form an open cover of $\mathcal{T}$.
\end{proposition}

\begin{proof}
Fix an arc triple $\partial= \{\partial_1,\partial_2,\partial_3\}$ and let $\eta$ be a positive number.
By Lemma \ref{lemma:infinity}, there exists a positive number $\eta_{\partial}$ such that for any point $h\in\mathcal{T}$ satisfying $\inf_{j\in\{1,2,3\}}l_{h}(\partial{j})\leq\eta_{\partial}$, there exists another arc triple $\delta_{h}$ with $h\in\mathcal{T}^{\eta}_{\delta}$.
Saying it differently, 
$$
\forall\partial,\ \forall\eta>0,\ \exists\eta_{\partial}>0\ \vert\ \mathcal{T}\setminus\mathcal{T}^{\eta_{\partial}}_{\partial}\subset\cup_{\delta}\mathcal{T}^{\eta}_{\delta}.
$$
By taking the smallest $\eta_{\partial}$ for all arc triples $\partial$, we conclude that
$$
\forall\eta>0,\exists\eta'>0\ \vert\ \forall\partial,\ \mathcal{T}\setminus\mathcal{T}^{\eta'}_{\partial}\subset\cup_{\delta}\mathcal{T}^{\eta}_{\delta}.
$$
Hence, 
$$
\forall\eta>0,\exists\eta'>0\ \vert\ \mathcal{T}\subset\cup_{\delta}\mathcal{T}^{\eta}_{\delta}\cup_{\partial}\mathcal{T}^{\eta'}_{\partial}.
$$
Fix $\eta>0$ and set $\epsilon:=\inf\{\eta,\eta'\}$.
Thus, $ \mathcal{T}\subset\cup_{\partial}\mathcal{T}^{\epsilon}_{\partial}$, which is what was to be shown.
\end{proof}

Let $\partial= \{\partial_1,\partial_2,\partial_3\}$ be an arc triple and let $\epsilon$ be a positive number.
We consider the map 
\[q_\partial : \mathcal{T}_\partial^\epsilon\hookrightarrow\mathcal{MF}_\partial\]
defined by foliating the right hexagon by leaves equidistant to certain arcs that do not belong to the arc triple $\partial$ in such a way that the transverse measure of each arc $\partial_i$ ($i=1,2,3$) coincides with its length, that is,
$l_h(\partial_i)=i(q_\partial^\epsilon (h),\partial_i)>\epsilon$ for $i=1,2,3$.
The injectivity of $q_\partial$ follows easily from Proposition \ref{prop:hexagon}.

\subsection{Compactification of Teichm\"uller space}

We denote by $\stackrel{\circ}{\mathcal{PMF}_\partial}$ the interior of $\mathcal{PMF}_\partial$. We describe local charts for $\overline{\mathcal{T}}$ around each point of $\stackrel{\circ}{\mathcal{PMF}_\partial}$, and we show that this space is a manifold with boundary.

Consider the map
\[\phi_\partial:\mathcal{T}_\partial^\epsilon\cup \stackrel{\mathrm{o}}{\mathcal{PMF}_\partial}
   \to\stackrel{\mathrm{o}}{\mathcal{PMF}_\partial} \times [0,1[\]
defined by
\[\displaystyle x\mapsto  
\begin{cases} (x,0) & \hbox { if } x\in\mathcal{PMF}_\partial \\
\left(\pi\circ q_\partial (x),e^{-\left(\partial_1+\partial_2+\partial_3\right)}\right) & \hbox { if } x\in\mathcal{T}_\partial^\epsilon.
\end{cases}\]
(Here, as elsewhere, we write $\partial_1$ instead of $l_x(\partial_1)$, etc.)
This maps defines, as the following propositions indicates, local coordinates for $\overline{\mathcal{T}}$ (see Figure \ref{pro}).
\begin{figure}[ht!]
\psfrag{1}{$\mathcal{PMF}_\partial$}
\psfrag{2}{
}
\psfrag{c}{$c$}
\psfrag{A}{$A$}
\psfrag{B}{$B$}
\psfrag{C}{$C$}
\psfrag{x}{$x$}
\psfrag{y}{$y$}
\psfrag{p}{$\alpha$}
\centering
\includegraphics[width=.6\linewidth]{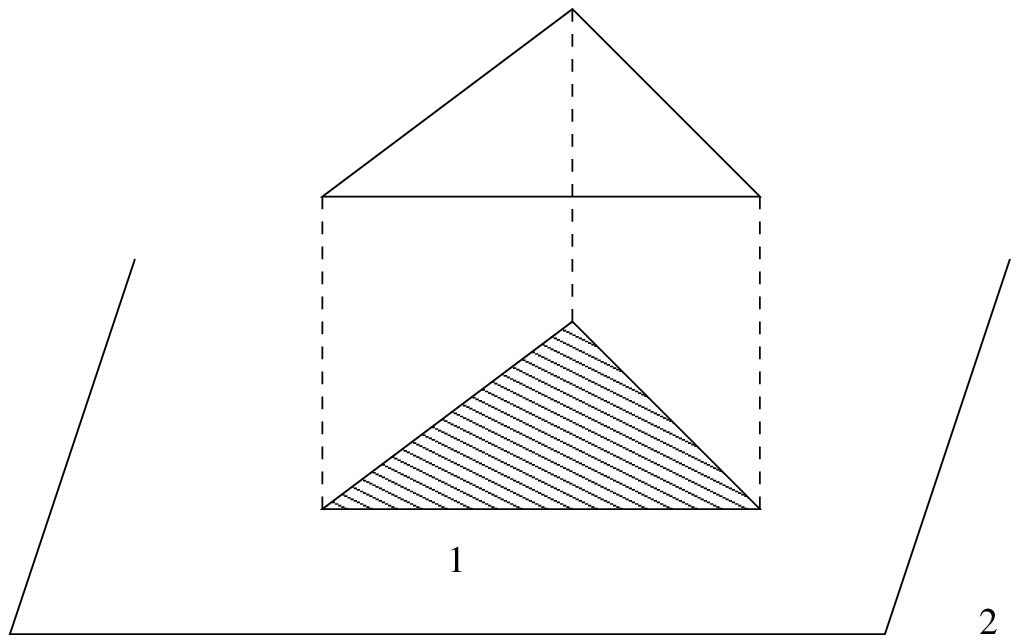}
\caption{\small{}}
\label{pro}
\end{figure}

\begin{proposition}
The map $\phi_\partial$ is a homeomorphism.
\end{proposition}

\begin{proof}
It is clear that a sequence $(x_n)$ in $\mathcal{T}_\partial^\epsilon$ converges to infinity if and only if $\partial_1+\partial_2+\partial_3\to\infty$. 

The projection $q_\partial$ is continuous. We deduce that if the sequence $(x_n)$ converges to a point $x\in \mathcal{PMF}_\partial$, then $\phi_\partial(x_n)\to\phi_{\partial}(x)$. Thus, $\phi_\partial$ is continuous.

Let us show that $\phi_\partial$ is injective.
Let $x,y\in \mathcal{T}_\partial^\epsilon$ be two points having the same image by  $\pi\circ q_\partial$ and such that
$\sum_{i=1}^3 l_x(\partial_i)= \sum_{i=1}^3 l_y(\partial_i)$.
Since $\pi\circ q_\partial(x)=\pi\circ q_\partial(y)$, there exists a non-zero real number $\lambda$ such that $q_\partial (x)=\lambda q_\partial (y)$. But then the last assumption implies $\lambda=1$, and since $q_\partial$ is injective, we get $x=y$.

Let us show that the map $\phi_\partial$ is invertible. Consider a continuous section $\sigma_\partial :\, \stackrel{\mathrm{o}}{\mathcal{PMF}_\partial}\to \stackrel{\mathrm{o}}{\mathcal{MF}_\partial}$. Up to multiplying this section by a scalar, we can assume that $i(\partial_j,\sigma_{\partial}(F))>\epsilon$ for $j=1,2,3$. 
Therefore there exists $x\in \mathcal{T}_\partial^\epsilon$ such that $q_\partial (x)=\sigma_\partial (F)$.
The point $x$ is now well defined up to homothety, and the constant factor is determined as soon as we know the value of the sum $\partial_1+\partial_2+\partial_3$. 
The inverse of $\phi_\partial$ is thus well defined.

To see that $\phi_\partial^{-1}$ is continuous, let $(z_n,t_n)\to (z,0)$. Since $t_n\to 0$, up to extracting a subsequence, there exists $\partial_i\in\partial$ that converges to infinity. Thus, the sequence $\phi_\partial^{-1}(z_n,t_n)$ converges to infinity, and $\pi\circ q (\phi_\partial^{-1}(z_n,t_n))=z_n\to z$ as $n\to\infty$. Thus, 
$\phi_\partial^{-1}(z_n,t_n)\to \phi_\partial^{-1}(z,0)$.

This completes the proof of the proposition.
\end{proof}

Since $\mathcal{T}=\cup_\partial \mathcal{T}_\partial^{\epsilon}$ for some $\epsilon>0$, we have shown that $\overline{\mathcal{T}}$ is a manifold with boundary, bounded by $\mathcal{PMF}$. 
  
We know that $\mathcal{PMF}$ is homeomorphic to a sphere.
The interior of a collar neighborhood of $\mathcal{PMF}$ is a sphere embedded in the interior of $\mathcal{T}$, which is homeomorphic to a ball. From Sch\"onflies' theorem, this sphere bounds a ball. Thus,  $\overline{\mathcal{T}}$ is homeomorphic to a three-dimensional closed ball.   
Let us summarize this in the following

\begin{theorem}
The Teichm\"uller space $\mathcal{T}$ of a right hexagon is homeomorphic to an open ball of dimension three and can be compactified as a closed ball of dimension three by the space of projective measured foliations $\mathcal{PMF}$ which consists of the bordary two-sphere.
\end{theorem}

We close this section by giving in Figure \ref{ex} a few examples of sequences of right hexagons converging towards the boundary, together with their limit points.

\begin{figure}[ht!]
\psfrag{a}{$a$}
\psfrag{b}{$b$}
\psfrag{c}{$c$}
\psfrag{A}{$A$}
\psfrag{B}{$B$}
\psfrag{C}{$C$}
\psfrag{1}{$x_n=$}
\psfrag{2}{$n\to\infty$}
\psfrag{3}{$a=n$}
\psfrag{4}{$b=n$}
\psfrag{5}{$c=n$}
\psfrag{t}{$1/3$}
\psfrag{m}{$1$}
\psfrag{n}{$0$}
\psfrag{20}{$\in$}
\psfrag{36}{$1/2$}
\psfrag{37}{$1/2$}
\psfrag{39}{$0.4$}
\psfrag{40}{$0.6$}
\psfrag{6}{$1/n$}
\psfrag{7}{$1/n$}
\psfrag{8}{$1/n$}
\psfrag{9}{$e^n$}
\psfrag{10}{$n$}
\psfrag{11}{$n$}
\psfrag{12}{$e^n$}
\psfrag{13}{$e^n$}
\psfrag{14}{$n$}
\psfrag{15}{$2e^n$}
\psfrag{16}{$3e^n$}
\psfrag{17}{$n$}
\centering
\includegraphics[width=.9\linewidth]{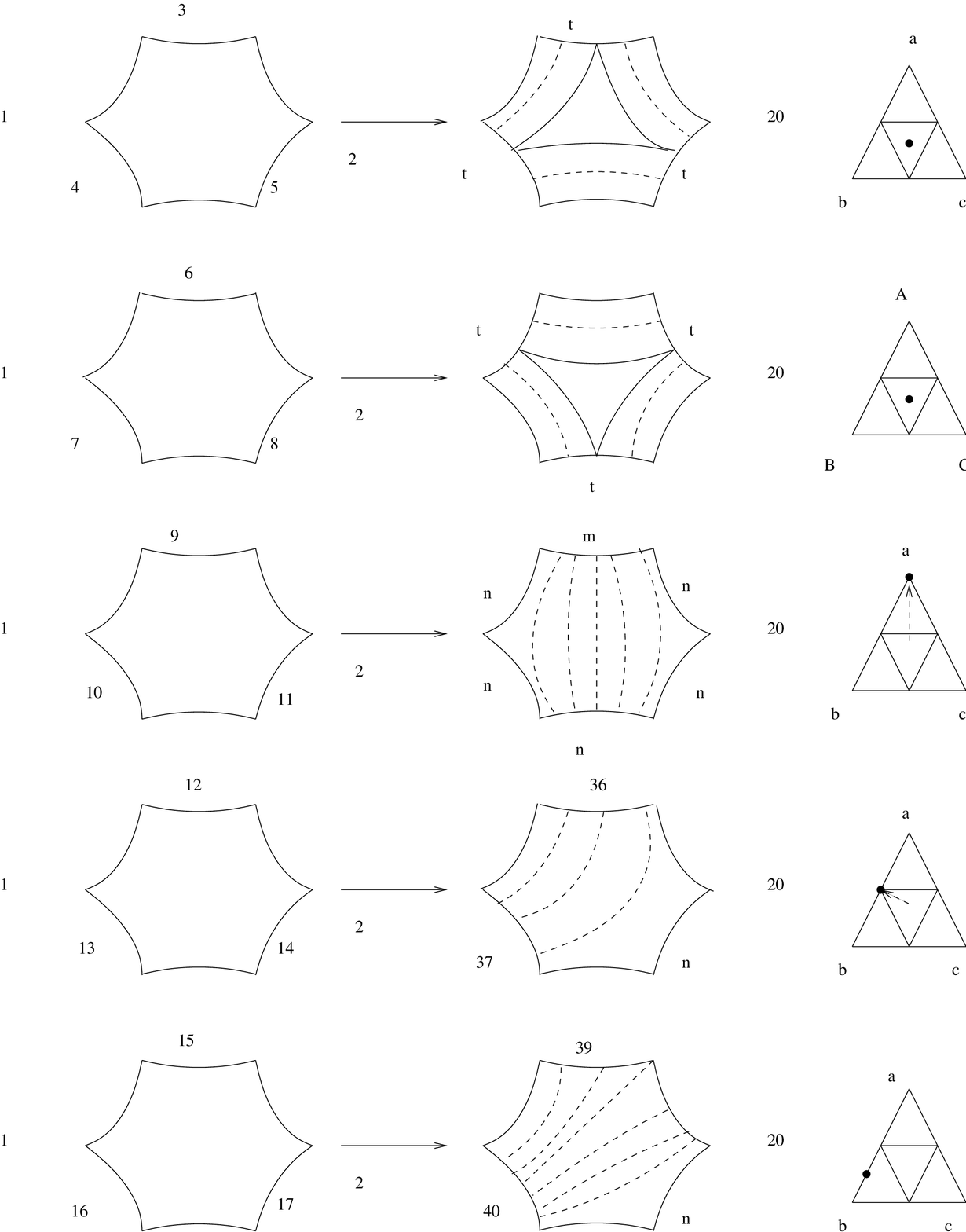}
\caption{\small{}}
\label{ex}
\end{figure}

   \section{The hyperbolic pair of pants}
   
   Any hyperbolic pair of pants is obtained by gluing two congruent right hexagons. Likewise, any measured foliation on a pair of pants has an order-two symmetry and is therefore obtained by gluing two foliated hexagons. Thus, we have 
   $\mathcal{T}(\hbox{pair of pants})=\mathcal{T}(\hbox{hexagon})$ and
      $\mathcal{MF}(\hbox{pair of pants})=\mathcal{MF}(\hbox{hexagon})$, and the theory for both surfaces is the same.
      
      Each arc triple for the hexagon corresponds to an element of the set $\mathcal{A}$ of boundary components and arcs joining boundary components, see Figure \ref{pants}.
   
\begin{figure}[ht!]
\psfrag{a}{$a$}
\psfrag{b}{$b$}
\psfrag{c}{$c$}
\psfrag{A}{$A$}
\psfrag{B}{$B$}
\psfrag{C}{$C$}
\psfrag{p}{$\alpha$}
\psfrag{q}{$\gamma$}
\centering
\includegraphics[width=.9\linewidth]{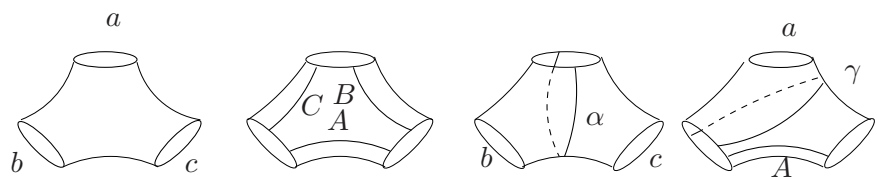}
\caption{\small{Decomposition of pairs of pants by hexagons.}}
\label{pants}
\end{figure}

\vfill\eject

\end{document}